\documentclass{amsart}
\usepackage{amsmath, amssymb, mathrsfs, verbatim, multirow}

\usepackage[all]{xy}
\usepackage{pifont}
\usepackage{float}
\usepackage{color}
\usepackage{enumitem}

\usepackage{tikz}
\usepackage{tikz-cd}
\usepackage{tkz-tab}
\usepackage{subcaption}

\usepackage{nicefrac}

\usepackage[hidelinks]{hyperref}

\newtheorem{Teo}{Theorem}[section]
\newtheorem{Prop}[Teo]{Proposition}
\newtheorem{Lema}[Teo]{Lemma}
\newtheorem{Cor}[Teo]{Corollary}

\newtheorem{Def}[Teo]{Definition}

\newtheorem{Obs}[Teo]{Remark}

\newtheorem{Exa}[Teo]{Example}

\newcommand{\Q}{\mathbb{Q}}
\newcommand{\R}{\mathbb{R}}
\newcommand{\Z}{\mathbb{Z}}
\newcommand{\N}{\mathbb{N}}

\newcommand{\F}{\mathbb{F}}

\newcommand{\Llr}{\Longleftrightarrow}
\newcommand{\lra}{\longrightarrow}

\newcommand{\VR}{\mathcal{O}}

\newcommand{\MI}{\mathfrak{m}}

\newcommand{\supp}{\mbox{\rm supp}}
\newcommand{\SU}{\mbox{\rm supp}}

\newcommand{\Appr}{\boldsymbol{ \rm A}}

\newcommand{\apprv}{ {\rm appr}_\nu(x,K)}

\begin{document}
\title{Parametrizations of subsets of the space of valuations}

\author{Josnei Novacoski}
\address{Departamento de Matem\'{a}tica,         Universidade Federal de S\~ao Carlos, Rod. Washington Luís, 235, 13565--905, S\~ao Carlos -SP, Brazil}
\email{josnei@ufscar.br}

\author{Caio Henrique Silva de Souza}
\address{Departamento de Matem\'{a}tica,         Universidade Federal de S\~ao Carlos, Rod. Washington Luís, 235, 13565--905, S\~ao Carlos -SP, Brazil}
\email{caiohss@estudante.ufscar.br}

\thanks{During the realization of this project the authors were supported by a grant from Funda\c c\~ao de Amparo \`a Pesquisa do Estado de S\~ao Paulo (process numbers 2017/17835-9 and 2021/13531-0).}

\begin{abstract} 
In this paper we present different ways to parametrize subsets of the space of valuations on $K[x]$ extending a given valuation on $K$. We discuss the methods using pseudo-Cauchy sequences and approximation types. The method presented here is slightly different than the ones in the literature and we believe that our approach is more accurate.
\end{abstract}


\maketitle

\section{Introduction}
In this paper we present the relation between \textit{pseudo-Cauchy sequences}, \textit{approximation types} and \textit{extensions of valuations}. More precisely, for a valued field $(K,v)$, we define the sets $\mathbb S$ of pseudo-Cauchy sequences in $K$, $\mathbb{A}$ of approximation types over $K$ and $\mathbb V$ of classes of valuations on $K[x]$ whose restriction to $K$ is equivalent to $v$. Then we present natural maps
\[
\Psi: \mathbb S\lra \mathbb V, \Phi: \mathbb A\lra \mathbb V\mbox{ and }\iota: \mathbb S\lra \mathbb A
\]
such that $\Phi\circ\iota=\Psi$.

Pseudo-Cauchy sequences (also called pseudo-convergent sequences) were introduced by Ostrowski (in \cite{Ost}) and Kaplansky (in \cite{Kapl}) in order to understand the possible extensions of $v$ to simple extensions (both algebraic and transcendental) of $K$. However, different pseudo-Cauchy sequences can define the same extension. Another type of object used for the same purpose, studied by Kuhlmann in \cite{KuhlmannApprTypes}, are approximation types. One advantage of these objects is that different approximation types define different valuations.

In this paper we review the objects mentioned above. However, we propose a slightly different approach. The approaches by Kaplansky and Kuhlmann allow us to present this map only in particular cases (for instance, if $K$ is algebraically closed). However, neither of these objects are enough to parametrize all the valuations on $\mathbb V$ (in the general case).

By Mac Lane and Vaquié's approach (\cite{MacLane}, \cite{Vaq} and \cite{Vaq2}), every valuation $\mu$ on $K[x]$ can be obtained by a \textit{chain of augmentations} of one of the following forms:
\begin{description}
\item[(a)] $\mu_0\lra \mu_1\lra\ldots\lra\mu_n= \mu$,
\item[(b)] $\mu_0\lra \mu_1\lra\ldots\lra\mu_n\lra \mu$, or
\item[(c)] $\mu_0\lra \mu_1\lra\ldots\lra\mu_n\lra \ldots\lra \mu$.
\end{description}
For more details on the above notation and results, we recommend \cite[Section 5]{NartTree}. The valuations that can be obtained using Kaplansky's or Kuhlmann's methods are those of the forms
\[
\mu_0\lra \mu\mbox{ (i.e., of type \textbf{(b)} and }n=0)
\]
or
\[
\mu_0\lra \mu_1=\mu\mbox{ (i.e., of type \textbf{(a)} and }n=1)
\]
where $\mu_0\lra \mu_1$ is a limit augmentation of a suitable form.

One recent method, presented in \cite[Theorem 3.2]{NartJosnei}, allows us to parametrize all the valuations on $\mathbb V$ by ``geometric objects". More precisely, they consider a parametrization of $\mathbb{V}$, which is a bijection, by the set $\mathbb D$ of equivalence classes of \textit{nests of discoids}. One can embed $\mathbb A$ into $\mathbb D$ in a natural way. The map $\Phi$ defined here can be seen as the restriction of the isomorphism $\mathbb D\lra \mathbb V$ to $\mathbb A$.

Another important approach to parametrize a large subset of $\mathbb V$ is presented in \cite{Per}. There the authors define the concept of \textit{pseudo-monotone sequences}, which generalizes the concept of pseudo-Cauchy sequences. With this approach, they present an association, similar to our map $\Psi$, that parametrizes every valuation on $\mathbb V$ when the completion of $K$ is algebraically closed.

The main tool used here are increasing sequences of valuations. There exists a group extension $\Lambda$ of the value group of $v$ for which every class in $\mathbb V$ contains a valuation $\nu:K[x]\lra \Lambda_\infty$, extending $v$. If we consider the set $\mathcal V$ of all such valuations, then we can consider an order on $\mathcal V$ by setting
\[
\nu\leq \mu \Llr \nu(f)\leq \mu(f)\mbox{ for every }f\in K[x].
\]
Then we can consider increasing families of valuations in $\mathcal V$ and study their suprema. The maps we present here are obtained by simply taking the supremum over a naturally obtained increasing family in $\mathcal V$.

Each element of $\mathbb V$ is exclusively of one of the following types: \textit{non-trivial support}, \textit{residue-transcendental}, \textit{value-transcendental} or \textit{valuation-algebraic} (see Section \ref{Preli}). We denote by $\mathbb V_{\rm nt}$, $\mathbb V_{\rm rt}$, $\mathbb V_{\rm vt}$ and $\mathbb V_{\rm al}$ the respective subsets of $\mathbb V$. The valuations with non-trivial support correspond to extensions of $v$ to simple algebraic extensions of $K$. Kaplansky's method allows us to associate pseudo-Cauchy sequences to extensions of $v$. For this purpose, he distinguishes between sequences of \textit{transcendental} or \textit{algebraic} types. The sequences of transcendental type are associated to valuations in $\mathbb{V}_{\rm al}$ and the sequences of algebraic type to valuations in $\mathbb V_{\rm nt}$ (i.e., extensions to algebraic extensions).

Consider
\[
{\rm Irr}(K)=\{f\in K[x]\mid f\mbox{ is monic and irreducible}\}.
\]
There exists a bijection between $\mathbb V_{\rm nt}$ and ${\rm Irr}(K^h)$ where $K^h$ denotes the \textit{henselization} of $K$ (see for instance \cite[Theorem 3.3]{NartTree}). In particular, if $K$ is algebraically closed, then there exists a bijection between these sets and $K$. However, if $K$ is not algebraically closed, then this set can be very large, so one cannot expect to parametrize every valuation in $\mathbb V_{\rm nt}$ by a pseudo-Cauchy sequence.

The maps presented here coincide with Kaplansky's in some cases, for instance for sequences of transcendental type. However, for some pseudo-Cauchy sequences of algebraic type, the valuations obtained are distinct. We list a few differences between these two approaches:
\begin{enumerate}
\item The map presented here is ``canonical" and Kaplansky's is not. More precisely, for a pseudo-Cauchy sequence of algebraic type, the extension by Kaplansky depends on a choice of a polynomial of smallest degree \textit{not fixed} by it.

\item The map presented here is defined in a unified way. Indeed, in all the cases, the map is obtained by taking the supremum of a naturally defined increasing family of valuations.

\item  The map $\Psi$ presented here allows us to naturally divide $\mathbb S$ into subsets
\[
\mathbb S=\mathbb S_{\rm nt}\sqcup\mathbb S_{\rm rt}\sqcup\mathbb S_{\rm vt}\sqcup\mathbb S_{\rm al}
\]
whose images lie in the respective subsets of $\mathbb V$.

\item We extend the definition of pseudo-Cauchy sequences so that we can also obtain \textit{depth-zero} valuations, i.e. of type \textbf{(a)} and $n=0$.

\item In the case where $K$ is algebraically closed, our map coincides with the one by Kaplansky.
\end{enumerate}  

The method used by Kuhlmann is very similar to the one by Kaplansky. Hence, all the differences mentioned above also apply to the map $\Phi$. Moreover, the definition of approximation type presented here does not coincide with Kuhlmann's. More precisely, in his approach, the \textit{balls} appearing in the approximation type can be both open or closed. In our definition, only closed balls are considered. To each closed ball (of the form $B(a,\gamma)$ as in Section \ref{constructuon}) we associate a depth-zero valuation (namely $v_{a,\gamma}$ as in Section \ref{constructuon}) and take the supremum over these valuations. If $K$ is algebraically closed, then by \cite[Lemma 6.2]{Andrei} we have
\[
v_{a,\gamma}(f)=\min_{b\in B(a,\gamma)}f(b)\mbox{ for every }f\in K[x].
\]
In particular, open balls do not play a role in this association.
 
This paper is divided as follows. In Section \ref{Preli} we present some basic results about valuations. In Section \ref{SecV}, we review the construction and properties of the universal group $\Lambda$. In Section \ref{Secincreafamil} we review general properties of increasing families of valuations. Finally, the maps $\Psi$ and $\Phi$ are introduced in Sections \ref{PseudoCauchySeq} and \ref{approcxsec}. In these sections we prove their main properties.
\par\medskip

\section{Preliminaries}\label{Preli}


\begin{Def}
For a commutative ring $R$, a \index{Valuation}\textbf{valuation} on $R$ is a map
\[
\nu:R\lra \Gamma_\infty :=\Gamma \cup\{\infty\}
\]
where $\Gamma$ is a totally ordered abelian group (and the extension of addition and order to $\infty$ is done in the natural way), with the following properties:
\begin{description}
	\item[(V1)] $\nu(ab)=\nu(a)+\nu(b)$ for all $a,b\in R$.
	\item[(V2)] $\nu(a+b)\geq \min\{\nu(a),\nu(b)\}$ for all $a,b\in R$.
	\item[(V3)] $\nu(1)=0$ and $\nu(0)=\infty$.
\end{description}
\end{Def}

\vspace{0.2cm} 	

Let $\nu: R \lra\Gamma_\infty$ be a valuation. The set
\[
\supp(\nu)=\{a\in R\mid \nu(a )=\infty\}
\]
is called the \textbf{support} of $\nu$. The \textbf{value group} of $\nu$ is the subgroup of $\Gamma$ generated by $\{\nu(a)\mid a \in R\setminus \supp(\nu) \}$ and is denoted by $\nu R$ or $\Gamma_\nu$. A valuation $\nu$ is a \index{Valuation!Krull}\textbf{Krull valuation} if $\supp(\nu)=\{0\}$.  If $\nu$ is a Krull valuation, then $R$ is a domain and we can extend $\nu$ to $ K={\rm Quot}(R)$ on the usual way. In this case, 
 define the \textbf{valuation ring} as $\VR_\nu:=\{ a\in K\mid \nu(a)\geq 0 \}$. The ring $\VR_\nu$ is a local ring with unique maximal ideal $\MI_\nu:=\{a\in K\mid \nu(a)>0 \}. $ We define the \textbf{residue field} of $\nu$ to be the field $\VR_\nu/\MI_\nu$ and denote it by $ k_\nu$. The image of $a\in \VR_\nu$ in $ k_\nu$ is denoted by $a\nu$. 
 
 \vspace{0.2cm} 	

\begin{Obs}Take  a valuation $\nu$ on a field $ K$ and  a valuation  $\overline{\nu}$ on $\overline{ K}$, the algebraic closure of $ K$, such that $\overline{\nu}|_{ K}=\nu$. Then $\overline{\nu}\overline{ K}$ is a divisible group. Additionally, $\overline{\nu}\overline{ K}=\nu K\otimes_\Z \Q$ (see \cite{Eng}, p.79). The group $\nu K\otimes_\Z \Q$ is called the \textbf{divisible closure} of $\nu K$ and is also denoted by $(\Gamma_{\nu})_{\Q}$. It is known that $k_{\overline{\nu}} $ is the algebraic closure of $k_\nu$ (see \cite{Eng}, p.66).

\end{Obs}

Take
$\mu: R\longrightarrow \Lambda_\infty \text{ and } \mu': R\longrightarrow \Lambda^{'}_\infty  $ two
valuations on a ring $R$. We say that $\mu$ and $\mu'$ are \textbf{equivalent} if there is an order-preserving isomorphism $\tau: \mu R \lra \mu'R$ such that $\tau \circ \mu =\mu'$. 

\subsection{Constructing valuations}\label{constructuon}
Take $(K,v)$ a valued field. For simplicity, we will denote its residue field by $k$ and its value group by $\Gamma_v$. We will describe a universal way to build new valuations on $K[x]$, extending $v$. This method was used implicitly in many important results about valuations (for instance, in \cite{Kapl}, \cite{MacLane} and \cite{Vaq}).

For a given $d\in\N$ we define
\[
K[x]_d:=\{f(x)\in K[x]\mid\deg(f)<d\}.
\]
Consider the following situation.
\begin{equation}                           \label{sit}
\left\{\begin{array}{ll}
\mu & \mbox{is a map from }K[x]_d\mbox{ to }\Gamma_\infty\mbox{ for some group }\Gamma\\
q & \mbox{is a polynomial on }K[x]\mbox{ of degree }d\\
\gamma & \mbox{is an element in a group extension }\Gamma'\mbox{ of }\Gamma
\end{array}\right. .
\end{equation}
For every polynomial $f\in K[x]$ we can write uniquely
\[
f=f_0+f_1q+\ldots+f_rq^r\mbox{ with }f_i\in K[x]_d\mbox{ for every }i, 0\leq i\leq r.
\]
This expression is called the \textbf{$q$-expansion of $f$}. Under the situation \eqref{sit}, we can define a map
\[
\mu_{q,\gamma}:K[x]\lra \Gamma'_\infty
\]
by
\[
\mu_{q,\gamma}\left(f_0+f_1q+\ldots+f_rq^r\right)=\min_{0\leq i\leq r}\{\mu(f_i)+i\gamma\}.
\]
We want to impose conditions on $\mu$, $q$ and $\gamma$ so that $\mu_{q,\gamma}$ is a valuation. For this purpose we present a variation of conditions \textbf{(V1)} and \textbf{(V2)} of valuations.

Let $S$ be a subset of a ring $R$, $\Gamma$ an ordered abelian group and $\eta:S\lra \Gamma_\infty$ any map. We say that $\eta$ satisfies \textbf{(V1')}  if for every $f,g\in S$, if $fg\in S$, then
	\[
	\eta(fg)=\eta(f)+\eta(g).
	\]
We say that $\eta$ satisfies \textbf{(V2')}  if for every $f,g\in S$, if $f+g\in S$, then
	\[
	\eta(f+g)\geq\min\{\eta(f),\eta(g)\}.
	\]

\begin{Lema}\cite[Lemma 2.3]{josneimonomial}\label{Lemasobreextsvelidesitra}
Under the situation \eqref{sit}, if $\mu$ satisfies \textbf{(V2')}, then $\mu_{q,\gamma}$ also satisfies \textbf{(V2')}.
\end{Lema}

\begin{Teo}\cite[Theorem 1.1]{josneimonomial}\label{lemasobreooutrocoiso}
Assume that we are in the situation \eqref{sit} and let $S$ be a subset of $K[x]$ closed under multiplication with $K[x]_d\subseteq S$. Assume that $\mu:S\lra \Gamma_\infty$ satisfies \textbf{(V2')} and that for every $\overline f,\overline g\in K[x]_d$ we have
	\begin{description}
		\item[(i)] $\mu(\overline f\overline g)=\mu(\overline f)+\mu(\overline g)$; and
		\item[(ii)] if $\overline f\overline g=aq+c$ with $c\in K[x]_d$ (and consequently $a\in K[x]_d$), then
		\[
		\mu(c)=\mu(\overline f\overline g)<\mu(a)+\gamma.
		\]
	\end{description}
	Then $\mu_{q,\gamma}$ satisfies \textbf{(V1')} and \textbf{(V2')}.
\end{Teo}

For $a\in K$ and $\gamma\in \Gamma'$, we denote $v_{a,\gamma}$ instead of $v_{x-a,\gamma}$. The following result follows immediately from Theorem \ref{lemasobreooutrocoiso} (for $\mu=v$, $q=x-a$, $d=1$ and $K[x]_d=K$).

\begin{Cor}
	For any $a\in K$ and $\gamma\in \Gamma'$, the map 
	$$v_{a,\gamma}(a_0+a_1(x-a)+\ldots+a_r(x-a)^r):=\min_{0\leq i\leq r}\{v(a_i)+i\gamma\} $$
	is a valuation on $K[x]$. 
\end{Cor}
The valuations $v_{a,\gamma}$ as defined above are called \textbf{monomial valuations}.

\begin{Def}
For $a\in K$ and $\gamma\in \Gamma_v$ the \textbf{(closed) ball centered at $a$ and with radius $\gamma$} is defined as
\[
B(a,\gamma)=\{b\in K\mid v(b-a)\geq \gamma\}.
\]
\end{Def}
\begin{Obs}
The concept of open ball is defined in a similar way:
\[
B^\circ(a,\gamma)=\{b\in K\mid v(b-a)> \gamma\}.
\]
This object will not play a role in this paper.
\end{Obs}
We present some easy properties of balls.
\begin{Lema}\label{lemasobrebola} $\,$
	
\begin{description}
\item[(i)] If $b\in B(a,\gamma)$, then $B(a,\gamma)=B(b,\gamma)$.
\item[(ii)] There exists $b\in B(a,\gamma)$ such that $v(b-a)=\gamma$.
\item[(iii)] If $B(a,\gamma)\subseteq B(a',\gamma')$, then $\gamma\geq\gamma'$.
\item[(iv)] If $B(a,\gamma)\subseteq B(a',\gamma')$ and $\gamma>\gamma'$, then $B(a',\gamma')\setminus B(a,\gamma)\neq \emptyset$.
\end{description}
\begin{proof}
In order to show \textbf{(i)}, assume that $b\in B(a,\gamma)$ and take $c\in B(b,\gamma)$. Then
\[
v(c-a)=v(c-b+b-a)\geq \min\{v(c-b),v(b-a)\}\geq\gamma.
\]
Hence $B(b,\gamma)\subset B(a,\gamma)$. Since $b\in B(a,\gamma)$ if and only if $a\in B(b,\gamma)$, the other inclusion follows by the symmetric argument.

We will now show \textbf{(ii)}. Take any $c\in B(a,\gamma)$. If $v(c-a)=\gamma$, then take $b=c$. If not, then we take any $d\in K$ such that $v(d)=\gamma$ and set $b=d+c$. Then
\[
v(b-a)=v(d+c-a)=v(d)=\gamma.
\]

If $B(a,\gamma)\subseteq B(a',\gamma')$, then $a\in B(a',\gamma')$ and by \textbf{(i)} we have $B(a,\gamma')=B(a',\gamma')$. By \textbf{(ii)} there exists $b\in B(a,\gamma)$ such that $v(b-a)=\gamma$. Since $B(a,\gamma)\subseteq B(a',\gamma')$, this shows that $b\in B(a,\gamma')$ and consequently
\[
\gamma=v(b-a)\geq \gamma'.
\]
This shows \textbf{(iii)}.

Finally, assume that $B(a,\gamma)\subseteq B(a',\gamma')$ and $\gamma>\gamma'$. By \textbf{(i)} we have $B(a,\gamma')=B(a',\gamma')$. By \textbf{(ii)}, there exists $b\in B(a,\gamma')$ such that $v(b-a)=\gamma'$. Since $\gamma>\gamma'$, this shows that $b\notin B(a,\gamma)$ and this completes the proof.
\end{proof}
\end{Lema}
The next result allows us to determine when two monomial valuations are comparable.
\begin{Lema}\label{lemNuaGammaMenorIgual}
	Take $a,a'\in K$ and $\gamma,\gamma'\in \Gamma_v$. Then the following conditions are equivalent.
\begin{description}
\item[(i)] $v_{a',\gamma'}\leq v_{a,\gamma}$,
\item[(ii)] $\gamma\geq \gamma'$ and $v(a'-a)\geq \gamma'$,
\item[(iii)] $B(a,\gamma)\subseteq B(a',\gamma')$.
\end{description}
\end{Lema}

\begin{proof}
Assume that $\gamma\geq \gamma'$ and $v(a'-a)\geq \gamma'$. Then
\[
v_{a,\gamma}(x-a')\geq\min\{\gamma,v(a'-a)\}\geq \gamma'.
\]
Hence, for any $f=a_0+a_1(x-a')+\ldots+a_r(x-a')^r$ we have
\[
v_{a,\gamma}(f)\geq \min_{0\leq i\leq r}\left\{v(a_i)+iv_{a,\gamma}(x-a')\right\}\geq \min_{0\leq i\leq r}\{v(a_i)+i\gamma'\}=v_{a',\gamma'}(f).
\]
This shows that \textbf{(ii)} implies \textbf{(i)}. For the converse, suppose that $\gamma< \gamma'$ or $v(a'-a)< \gamma'$. Then
\[
v_{a,\gamma}(x-a')=\min\{\gamma,v(a'-a)\}<\gamma'=v_{a',\gamma'}(x-a').
\]
Hence, $v_{a',\gamma'}\not\leq v_{a,\gamma}$.

Assume that \textbf{(ii)} is satisfied and take $b\in B(a,\gamma)$. Then
\[
v(b-a')=v(b-a+a-a')\geq \min\{v(b-a), v(a-a')\}\geq \gamma'. 
\]
Consequently, \textbf{(iii)} is satisfied. The converse follows from the definition of ball and Lemma \ref{lemNuaGammaMenorIgual} \textbf{(iii)}.
\end{proof}

For a valuation $\nu$ on $K[x]$ we denote by $\Gamma_\nu$ its valued group and (if $\SU(\nu)=0$) its residue field by $k_\nu$. Any valuation $\nu$ on $K[x]$, extending $v$, is (exclusively) of one of the following types:
\begin{itemize}
	\item\textbf{Nontrivial support}: $\SU(\nu)=fK[x]$ for some irreducible $f\neq 0$.
	\item\textbf{Value-transcendental}: $\SU(\nu)=0$ and $\Gamma_\nu/\Gamma_v$ is not a torsion group.
	\item\textbf{Residue-transcendental}: $\SU(\nu)=0$ and  $k_\nu/k$  is transcendental.
	\item\textbf{Valuation-algebraic}: $\SU(\nu)=0$,  $\Gamma_\nu/\Gamma_v$ is a torsion group and $k_\nu/k$  is algebraic.
\end{itemize}

\begin{Def}\label{defValAlg} A valuation $\nu$ on $K[x]$  extending $v$  is called \textbf{valuation-transcendental} if it is value-transcendental or residue-transcendental.
	
\end{Def}
\begin{Obs}
By Abhyankar's inequality (see \cite{zar}, p.330) a valuation cannot be  value-transcendental and residue-transcendental at the same time. 
\end{Obs}

The next result is well-known.
\begin{Lema}\cite[Proposition 1.1]{popescu2}\label{lemqnSobrecTransc1}
If $a\in K$ and $\gamma\in \Gamma_v$, then the valuation $v_{a,\gamma}$ is residue-transcendental.
\end{Lema}

\section{Universal value group for small extensions of $\Gamma_v$}\label{SecV}
We review now some results from \cite{NartTree} that will be helpful in what follows. For simplicity of notation, in this section we denote $\Gamma=\Gamma_v$. The main goal is to present a universal ordered group $\Lambda$ containing, up to order-preserving isomorphism, all \textit{small extensions}\footnote{$\Gamma\hookrightarrow\Lambda$ is a small extension of abelian ordered groups if $\Gamma/\Delta$ is a cyclic group, where $\Delta\subset \Lambda$ is the relative divisible closure of $\Gamma$ in $\Lambda$. } of $\Gamma$. In particular, this universal group has the property that each class of valuations whose restriction to $K$ is equivalent to $v$ contains a valuation $\mu:K[x]\lra \Lambda_\infty$.

Given $\gamma\in \Gamma$, the \textit{principal convex subgroup} of $\Gamma$ generated by $\gamma$ is the  intersection of all convex subgroups of $\Gamma$ containing $\gamma$. 
Let
$$I:= {\rm Prin}(\Gamma) $$
be the totally ordered set of non-zero principal convex subgroups of $\Gamma$, ordered by decreasing inclusion.  
We take the Hahn product
$$\R_{\text{lex}}^I:=\{ y=(y_i)_{i\in I}\mid \supp(y) \text{ is well-ordered}\} \subset \R^I $$
where $\supp(y) = \{ i\in I\mid y_i\neq 0  \}$. 

Let ${\rm Init}(I)$ be the set of initial segments of $I$. For each $S\in {\rm Init}(I)$, we consider a formal symbol $i_S$ and the ordered set
$$I_S:= I\sqcup\{i_S\}, $$
with the order defined by $i<i_S$ for all $i\in S$ and $i_S<j$ for all $j\in I\setminus S$. We define
$$\mathbb{I}:= I\sqcup \{ i_S\mid S\in {\rm Init}(I)\}. $$
The order in $\mathbb{I}$ is such that $I_S \hookrightarrow \mathbb{I}$ preserves the order for all $S\in  {\rm Init}(I)$ and, for every $S, T\in {\rm Init}(I)$, we have $i_S<i_T$ if and only if $S\subsetneq T$. As above, we consider the Hahn product $$\R_{\text{lex}}^{\mathbb{I}}\subset \R^{\mathbb{I}}. $$ For all $S\in  {\rm Init}(I)$, since $I\subset I_S \subset \mathbb{I}$,  we have the embeddings
$$ \R_{\text{lex}}^I \hookrightarrow \R_{\text{lex}}^{I_S}\hookrightarrow \R_{\text{lex}}^{\mathbb{I}}$$
Also, \cite{NartTree} shows that  $$ \Gamma \hookrightarrow \Gamma_\Q \hookrightarrow \R_{\text{lex}}^I \hookrightarrow \R_{\text{lex}}^{\mathbb{I}}.  $$
\begin{Prop}\cite[Proposition 6.2]{NartTree}
Let $\mu$ be a valuation on $K[x]$ whose restriction to $K$ is equivalent to $v$. Then, there exists an embedding $j:\Gamma_\mu \hookrightarrow \R_{\text{lex}}^{\mathbb{I}}$ satisfying the following properties:
\begin{description}
	\item[(i)] the following diagram commutes:

	\begin{center}
			\begin{tikzcd}[row sep=large, column sep = large]
		K[x] \arrow[r, rightarrow, "j\circ \mu"]{}{}  & (\mathbb{R}_{\rm lex }^{\mathbb{I}})_\infty \\
	K \arrow [u, hookrightarrow ]{}{} \arrow[r , "v", labels=below]{}{} & \Gamma_\infty \arrow [u, hookrightarrow ]{}{}
	\end{tikzcd}
	\end{center}
	
	\item[(ii)] There exists $S\in  {\rm Init}(I)$ such that $j(\Gamma_\mu)\subset \R_{\text{lex}}^{I_S}$. 
\end{description}
\end{Prop}

Hence, $\mu$ is equivalent to the valuation $j\circ \mu$ on $K[x]$ and $v$ is equivalent to $\l \circ v$ on $K$. We can then assume that $v$ takes values in $\R_{\text{lex}}^{I}$. Moreover,  for each class of valuations in $ \mathbb{V}$ we can take a representative $\mu$ such that $\mu|_K =v$ and $\Gamma_\mu\subset  \R_{\text{lex}}^{I_S}$ for some $S\in  {\rm Init}(I)$ (\cite{NartTree}, p.35).

We define
$$\R_{\text{sme}}:=  \bigcup_{S\in  {\rm Init}(I) } \R_{\text{lex}}^{I_S} \subset \R_{\text{lex}}^{\mathbb{I}}  $$
and
$$\mathcal{V} = \{\nu: K[x] \rightarrow (\R_{\text{sme}})_\infty \mid \nu \text{ is a valuation extending } v  \}. $$
From what we have seen above, each equivalence class of valuations $[\nu]$ in $\mathbb{V}$ has a representative in $\mathcal{V}$.

For each $\alpha\in \R_{\text{sme}}$, let $\langle \Gamma, \alpha \rangle \subset \R_{\text{lex}}^{\mathbb{I}}$ be the subgroup generated by $\Gamma$ and $\alpha$. We define the following equivalence relation on $\R_{\text{sme}}$: for $\alpha, \beta\in \R_{\text{sme}}$, 
$ \alpha \sim_{\rm sme} \beta$ if and only if there exists an isomorphism of ordered groups $ \langle \Gamma, \alpha \rangle \to \langle \Gamma, \beta \rangle   $ which sends $\alpha$ to $\beta$ and acts as the identity on $\Gamma$. Take any subset $\Gamma_{\rm sme}\subset \R_{\text{sme}}$ of representatives of the quotient set $\R_{\text{sme}}/\sim_{\rm sme}$. According to \cite{NartTree}, we have
$$\Gamma \subset \Gamma_\Q \subset \Gamma_{\text{sme}} \subset \R_{\text{sme}}. $$
  The set $\Gamma_{\text{sme}}$ is complete when equipped with the order topology (\cite[Corollary 6.7]{NartTree}), that is, every non-empty subset of $\Gamma_{\text{sme}}$ has a supremum and an infimum. Indeed, the following canonical description for $\Gamma_{\text{sme}}$ (described in \cite{KuhlmannNart}  and \cite{NartTree}) guarantees this property.
  
 \subsection{Quasi-cuts} 
 
 \begin{Def}A \textbf{quasi-cut} in $\Gamma_\Q$ is a pair $\delta = (\delta^L,\delta^R)$ of subsets of $\Gamma_\Q$ such that
 	$$\delta^L\leq \delta^R \text{ and } \delta^L\cup \delta^R = \Gamma_\Q. $$
 	
 \end{Def}

\begin{Obs}We collect bellow some properties of the quasi-cuts in $\Gamma_\Q$. 
	
	\begin{itemize}
		\item The subset $\delta^L$ is an initial segment of $\Gamma_\Q$ and $\delta^L\cap \delta^R$ has at most one element. 
		
		\item We denote by ${\rm Qcuts}(\Gamma_\Q)$ the set of all quasi-cuts in $\Gamma_\Q$. We can define a total order in ${\rm Qcuts}(\Gamma_\Q)$ by setting
		$$\delta = (\delta^L,\delta^R)\leq \gamma = (\gamma^L,\gamma^R) \Llr \delta^L\subseteq \gamma^L \text{ and } \delta^R\supseteq \gamma^R. $$
		
		\item There is an embedding $\Gamma_\Q\hookrightarrow {\rm Qcuts}(\Gamma_\Q) $ that preserves the order, mapping $\gamma\in \Gamma_\Q$ to $((\Gamma_{\Q })_{\leq \gamma},(\Gamma_{\Q})_{\geq \gamma})$,
		where
		$$(\Gamma_{\Q })_{\leq \gamma} = \{ \sigma\in \Gamma_{\Q} \mid \sigma\leq \gamma \} \text{ and } (\Gamma_{\Q })_{\geq \gamma} = \{ \sigma\in \Gamma_{\Q} \mid \sigma\geq \gamma \}. $$
		This is called the \textbf{principal quasi-cut} associated to $\gamma$. We use this identification and assume that $\Gamma_\Q \subset {\rm Qcuts}(\Gamma_\Q)$.
		
		\item If $\delta = (\delta^L,\delta^R)$ is such that $\delta^L\cap \delta^R = \emptyset$, then $\delta$ is called a \textbf{cut} in $\Gamma_\Q$. Calling ${\rm Cuts}(\Gamma_\Q)$ the set of all cuts in $\Gamma_\Q$, we have 
		\[
		{\rm Qcuts}(\Gamma_\Q) = \Gamma_\Q \sqcup {\rm Cuts}(\Gamma_\Q).
		\]
		
	\item If $\delta = (\delta^L,\delta^R)$ is a cut, then $\delta^L$ does not have maximum or $\delta^R$ does not have minimum. Indeed, if $m^L\in \Gamma_{\Q}$ is a maximum of $\delta^L$ and $m^R\in \Gamma_{\Q}$ is a minimum for $\delta^R$, then by the divisibility of $\Gamma_{\Q}$ we have $m=(m^L+m^R)/2\in \Gamma_{\Q}$ and $m^L<m<m^R$, which is a contradiction.

	\end{itemize}
\end{Obs}

Take a cut $\delta =(\delta^L,\delta^R)$ on $\Gamma_\Q$.  We consider a formal symbol $x_\delta$ and look to the abelian group $\Gamma_\Q(\delta)= x_\delta\Z\oplus \Gamma_\Q$. 

If $\delta^L$ does not have maximum, then we define the relation
$$mx_\delta +b\leq nx_\delta+a \Longleftrightarrow (m-n)I \leq a-b $$
for some non-empty  final segment $I\subset \delta^L$. 
On the other hand, if $\delta^L$ has a maximum, then 
we define the relation
$$mx_\delta +b\leq nx_\delta+a \Longleftrightarrow (m-n)I \leq a-b $$
for some non-empty  initial segment $I\subset \delta^R$.

\begin{Lema}
	The above relation on $\Gamma_\Q(\delta)$ is a total order which is compatible with the group structure and satisfies $\delta^L<x_\delta <\delta^R$. 
\end{Lema}

\begin{proof}
We suppose $\delta^L$ has no maximum. The other case is analogous. 
The reflexivity of the relation is immediate. For the antisymmetry, take $mx_\delta +b$ and $nx_\delta+a$ such that $$mx_\delta +b\leq nx_\delta+a\text{ and } nx_\delta +a\leq mx_\delta+b.$$ We need to show that $n=m$ and $a=b$.  There exist non-empty final segments $I,J\subset \delta^L$  such that $(m-n)I \leq a-b$ and $(n-m)J\leq b-a$. The intersection $I\cap J$ is also a non-empty final segment and has at least two elements, because $\delta^L$ has no maximum. For every $\gamma \in I\cap J$, we have
	$$(n-m)\gamma\leq b-a \leq -(m-n)\gamma=(n-m)\gamma, $$
	that is, $(n-m)\gamma=b-a$. Hence, for different $\gamma$ and $\gamma'$ in $I\cap J$ we have $(n-m)\gamma = (n-m)\gamma'$, i.e., $(n-m)(\gamma-\gamma')=0$. This can only happens if $n=m$, since $\Gamma_\Q$ is torsion free.  Thus, $a=b$. 
	
	For the transitivity,  if $mx_\delta +b\leq nx_\delta+a$ and $nx_\delta +a\leq tx_\delta+c$, then there exist non-empty final segments $I,J\subset \delta^L$  such that $(m-n)I \leq a-b$ and $(n-t)J\leq c-a$. Hence, for every $\gamma \in I\cap J$ we have 	
	$$(m-t)\gamma = (m-n)\gamma+(n-t)\gamma \leq a-b +c-a = c-b. $$
	That is, $mx_\delta +b\leq tx_\delta+c$.  
	
	For the compatibility, take $mx_\delta +b\leq nx_\delta+a$ and any $tx_\delta +c$. We have, for every $\gamma \in I$,
	\begin{align*}
		(m-n)\gamma \leq a-b &\Longleftrightarrow ((m+t)-(n+t))\gamma \leq (a+c)-(b+c) \\
		&  \Longleftrightarrow (m+t)x_\delta +(b+c)\leq (n+t)x_\delta+(a+c)\\
		& \Longleftrightarrow (mx_\delta +b)+(tx_\delta +c)\leq  (nx_\delta +a)+(tx_\delta +c).
	\end{align*}
	
	Now we see that this is a total order. Take any $mx_\delta +b$ and $nx_\delta +a$ in $\Gamma(\delta)$. If $m=n$, then the total order on $\Gamma_\Q$ gives us that $a\leq b$ or $b\leq a$, hence by compatibility $mx_\delta +b\leq nx_\delta+a$ or $nx_\delta +a\leq mx_\delta+b$.  Suppose  $m-n>0$. If $(a-b)/(m-n)\in \delta^R$, then we take $I = \delta^L$. Hence,  for every $\gamma\in I$ we have
	$$\gamma < \frac{a-b}{m-n} \Longleftrightarrow (m-n)\gamma < a-b \Longleftrightarrow mx_\delta +b < nx_\delta+a. $$
	If $(a-b)/(m-n)\in \delta^L$, then we take $I = \{\gamma\in \delta^L\mid \gamma > (a-b)/(m-n)\}$. Hence, for every  $\gamma\in I$ we have
	$$\gamma > \frac{a-b}{m-n} \Longleftrightarrow (m-n)\gamma > a-b \Longleftrightarrow (n-m)\gamma < b-a \Longleftrightarrow nx_\delta +a < mx_\delta+b. $$
	The case where $m-n<0$ is analogous.  
	
	Last, we check that $\delta^L<x_\delta <\delta^R$. Take $b\in \delta^L$. Since $\delta^L$ does not have a maximum, we consider $I =  \{\gamma\in \delta^L\mid \gamma>b\}$. 
	Hence, $b<\gamma$ for every $\gamma \in I$ and then $(0-1)\gamma <(0-b)$, which means $b<x_\delta$. Now take $a\in \delta^R$. Consider $I = \delta^L$.
	Hence, $a>\gamma$ for every $\gamma \in I$ and then $(1-0)\gamma <(a-0)$, which means $x_\delta<a$. 
	\end{proof}


Therefore, $x_\delta$ is a supremum for $\delta^L$ and an infimum for $\delta^R$ in $\Gamma(\delta)$. Also, note that no element $\gamma\in \Gamma_\Q$ can have the property $\delta^L<\gamma<\delta^R$, because we must have $\gamma\in \delta^L$ or $\gamma\in \delta^R$ (since $\Gamma_\Q = \delta^L\cup \delta^R$). 

\vspace{0.5cm}

For each $\alpha\in \R_{\text{sme}}$, let $\delta_\alpha$ be the quasi-cut defined by
\[
\delta_\alpha^L = \{ \gamma\in \Gamma_\Q\mid \gamma\leq \alpha  \}\mbox{ and }\delta_\alpha^R = \{ \gamma\in \Gamma_\Q\mid \gamma\geq \alpha  \}.
\]
The next lemma tells us that the equivalence relation $\sim_{\rm sme }$ can be described in terms of quasi-cuts.

\begin{Lema}\cite[Lemma 5.4]{KuhlmannNart}
	For all $\alpha,\beta\in \R_{\text{sme}}$, we have $\alpha \sim_{\rm sme } \beta$ if and only if $\delta_\alpha= \delta_\beta$. 
\end{Lema}

It is possible to write the following isomorphism of ordered sets (see Section 6.4 of \cite{NartTree}):
\begin{align*} \Gamma_{\text{sme}} &\longrightarrow {\rm Qcuts}(\Gamma_\Q)\\
\alpha	& \longmapsto \delta_\alpha.
\end{align*}
If  $\delta\in {\rm Cuts}(\Gamma_\Q)$, then we can consider the formal symbol $x_\delta$ as the $\alpha\in \Gamma_{\text{sme}}$ such that $\delta = \delta_\alpha$.

\section{Limits of families of valuations}\label{Secincreafamil}
In this section we review some results about increasing families of valuations extending a fixed valuation $v$ on $K$. Let $\Lambda=\R_{\text{sme}}$ be the group constructed in the previous section. By what we saw before, every class of valuations in $\mathbb{V}$ admits a representative in
$$\mathcal{V} = \{\nu: K[x] \rightarrow \Lambda_\infty \mid \nu \text{ is a valuation extending } v  \}. $$
Also, since all the valuations on $\mathcal{V}$ are mapped to a common group, $\mathcal{V}$ admits a natural partial order.
\begin{Prop}\cite[Section 4]{Rig}\label{treescufutue}
The partially ordered set $\mathcal V$ is a tree, i.e., for every $\mu\in \mathcal V$, the set
\[
\left]-\infty, \mu\right[:=\{\nu\in \mathcal V\mid \nu< \mu\}
\]
is totally ordered.
\end{Prop}

\begin{Prop}\cite[Corollary 2.3]{novbarnabe} Let $\{\nu_i\}_{i\in I}$ be a totally ordered set in $\mathcal{V}$. For every $f\in K[x]$, either $\{\nu_i(f)\}_{i\in I}$ is strictly increasing, or there exists $i_0\in I$ such that $\nu_i(f)=\nu_{i_0}(f)$ for every $i\in I$ with $i\geq i_0$.  
\end{Prop}

A polynomial $f\in K[x]$ is $\mathfrak{v}$\textbf{-stable} if there exists $i_f\in I$ such that $\nu_i(f)=\nu_{i_f}(f)$ for every $i\geq i_f$. In particular, $\nu_i(f)\leq \nu_{i_f}(f)$ for every $i\in I$.

\begin{Lema}\cite[Lemma 4.5]{josneicaio2}\label{lemMenorIgualMu} Let $\mathfrak{v}$ be an increasing family of valuations without largest element. Consider $\nu\in\mathcal{V}$ such that $\nu\geq \nu_i$ for every $\nu_i\in \mathfrak{v}$.  Fix $k\in I$. We have the following. 
	\begin{enumerate}
		
		\item If $\nu_k(f)<\nu(f)$, then $\nu_k(f)<\nu_j(f)$ for every $j>k$.
		
		\vspace{0.1cm} 
		
		\item We have $\nu_k(f) = \nu(f)$  if and only if $\nu_k(f)=\nu_j(f)$ for every $j\geq k$. In other words, $f$ is $\mathfrak{v}$-stable if and only if 
		$\nu_k(f)=\nu(f)$ for some $k\in I$.
	\end{enumerate}
\end{Lema}

\begin{Def}
	Let $\mathfrak{v}= \{\nu_i\}_{i\in I}$ be a totally ordered set in $\mathcal{V}$ and $\mu\in \mathcal{V}$. We say that $\mu$ is the\textbf{ supremum} of $\mathfrak{v}$ if $\nu_i\leq \mu$ for every $i\in I$ and if $\mu'\in \mathcal{V}$ is such that $\nu_i\leq \mu'$ for every $i\in I$, then  $\mu\leq\mu'$. We denote $\mu = \displaystyle \sup_{i\in I}\nu_i = \sup \mathfrak{v}$.
\end{Def}

\begin{Prop}\cite[Proposition 4.9]{josneicaio2}\label{lemVstable} Let $\mathfrak{v}= \{\nu_i\}_{i\in I}$ be a totally ordered set in $\mathcal{V}$
	such that 
	every $f\in K[x]$ is $\mathfrak{v}$-stable. 
	Define
\[
\nu_\mathfrak{v} : K[x] \lra \Lambda_\infty\mbox{ by }\nu_\mathfrak{v}(f) :=  \nu_{i_f}(f).
\] 
	Then $\nu_\mathfrak{v}\in \mathcal V$ and it is the supremum of $\mathfrak v$. 
	
\end{Prop}

\begin{Cor}\cite[Theorem 1.2 and Lemma 4.1]{NartTree}\label{CorVstableAlgebraic}
The valuation $\nu_{\mathfrak v}$ in Proposition~\ref{lemVstable} is valuation-algebraic. 
\end{Cor}



Consider the set 
$$C(\mathfrak{v}):=\{ f\in K[x]\mid f \text{ is } \mathfrak{v}\text{-stable }\}.$$
For every $f\in C(\mathfrak{v})$ we set $\nu_\mathfrak{v}(f) = \nu_{i_f}(f)$.

\begin{Lema}\label{lemCofinalEachOther}
	Suppose that $\mathfrak{v}$ and $\mathfrak{v}'$ are two increasing families of valuations such that $ \mathfrak{v}$ and $ \mathfrak{v}'$ are cofinal in each other. Then,
	\[
	C(\mathfrak{v}) =C(\mathfrak{v}')\mbox{ and }\nu_\mathfrak{v}(f)=\nu_{\mathfrak{v}'}(f)\mbox{ for every }f\in C(\mathfrak{v}).
	\]
\end{Lema}
\begin{proof}
	Write $\mathfrak{v}=\{\nu_i\}_{i\in I}$ and $\mathfrak{v}'=\{\mu_j\}_{j\in J}$ and take $f\in C(\mathfrak{v})$. Then, there exists $i_f$ such that $\nu_i(f)=\nu_{i_f}(f)$ for every $i\geq i_f$. By the cofinality of $\mathfrak v'$ in $\mathfrak v$, there exists $j_f\in J$ such that $\nu_{i_f}\leq \mu_{j_f}$. For any $j\geq j_f$, take $i\in I$ large enough such that  $\nu_i\geq \mu_{j}$ (in particular, $i\geq i_f$), which exists by the cofinality of $\mathfrak v$ in $\mathfrak v'$. Then
	\[
		\nu_i(f)= \nu_{i_f}(f)\leq \mu_{j_f}(f)\leq \mu_j(f)\leq \nu_i(f).
		\]
Then equality holds everywhere on the above inequalities. Consequently, $f\in C(\mathfrak{v}')$ and $\nu_{\mathfrak v}(f)=\nu_{\mathfrak v'}(f)$.
The other inclusion follows by the symmetric argument.
\end{proof}

Let $\mathfrak{v}= \{\nu_i\}_{i\in I}$ be a totally ordered set in $\mathcal{V}$ such that $\mathfrak{v}$ has no maximum and admits unstable polynomials. 
Let $Q$ be a monic $\mathfrak{v}$-unstable polynomial of smallest  degree among $\mathfrak{v}$-unstable polynomials and take $\gamma\in \Lambda_\infty$ such that $\gamma>\nu_i(Q)$ for every $i\in I$. 

Consider the map 
$$\mu_{Q,\gamma}(f)=\underset{0\leq j \leq r}{\min}\{ \nu_{\mathfrak{v}}(f_j)+j\gamma  \}, $$
where $f_0+f_1Q+\ldots+f_rQ^r$ is the $Q$-expansion of $f$. 

\begin{Prop}\label{propMuQGamma}\cite[Theorem 2.4]{novbarnabe}, \cite[Proposition 1.21]{Vaq}, \cite[Theorem 5.1]{MacLane}
	The map $\mu_{Q,\gamma}$ belongs to $\mathcal{V}$ and  $\nu_i<\mu$ for every $i\in I$. 
\end{Prop}

Consider the cut $\delta$ whose left cut set is
 $$\{\alpha\in \Gamma_{\Q }\mid \alpha \leq \nu_i(Q) \text{ for some } i \in I  \} .$$
Consider $\gamma= x_{\delta}$.

\begin{Prop}\label{lemSupAlgType}
The cut $\delta$ and the valuation $\mu_{Q,\gamma}$ are independent of the choice of the polynomial $Q$, among the monic $\mathfrak v$-unstable polynomials of smallest degree. Moreover,
\[
\mu_{Q,\gamma}=\sup_{i\in I}\nu_i.
\]
\end{Prop}

\begin{proof}
Suppose  that $Q'$ is another polynomial of minimal degree not $\mathfrak{v}$-stable and write $Q'=Q+h$. Since $\deg(Q)=\deg(Q')$ and both are monic, we have $\deg(h)<\deg(Q)$, so that $h\in C(\mathfrak v)$. Since $\{\nu_i(h)\}_{i\in I}$ is ultimately constant and $\{\nu_i(Q)\}_{i\in I}$ and $\{\nu_i(Q')\}_{i\in I}$ are increasing, we deduce that $\nu_i(Q)=\nu_i(Q')$ for sufficiently large $i\in I$. This implies that $\gamma$ does not depend on the choice of $Q$ among the monic $\mathfrak v$-unstable polynomials of smallest degree. The fact that $\mu_{Q,\gamma}=\mu_{Q',\gamma}$ will follow immediately from the fact that $\mu_{Q,\gamma}=\sup_{i\in I}\nu_i$.

By Proposition \ref{propMuQGamma}, in order to show that $\mu_{Q,\gamma}=\sup_{i\in I}\nu_i$ we have to show that if $\mu\in \mathcal{V}$ is such that $\mu\geq   \nu_i$ for every $i\in I$, then $\mu\geq \mu_{Q,\gamma}$. If $f\in K[x]$ is such that $\deg(f)<\deg(Q)$, then $f$ is  $\mathfrak{v}$-stable and for some $i\in I$ we have
\[
\nu_{Q,\gamma}(f)=\nu_i(f)\leq \mu(f).
\]
Also, since $\mu(Q)\geq \nu_{i}(Q)$ for every $i\in I$, we have
	$$\mu(Q)\geq \gamma = \mu_{Q,\gamma}(Q).  $$
For any $f\in K[x]$ let
\[
f=f_0+f_1Q+\ldots+f_rQ^r
\]
be the $Q$-expansion of $f$. Then
$$\mu(f)\geq\min_{0\leq i\leq r}\{\mu(f_i)+i\mu(Q)\}  =  \min_{0\leq i\leq r}\{\nu_{Q,\gamma}(f_i)+i\gamma\} = \mu_{Q,\gamma}(f).$$
Hence, $\mu_{Q,\gamma}=\sup_{i\in I}\nu_i. $
\end{proof}

\section{Parametrization of $\mathbb{V}$ via  pseudo-Cauchy sequences }\label{PseudoCauchySeq}

We begin this section with some definitions and results concerning pseudo-Cauchy sequences.

\begin{Def}
	Let $(K,v)$ be a valued field and  $\underline{a}=\{a_{i}\}_{i\in I}$ a subset of $K$, where $I$ is a well-ordered set with at least two elements. We say that  $\underline a$ is a \textbf{pseudo-Cauchy sequence} if
		\[
		v(a_i-a_j)<v(a_j-a_k)\mbox{ for every }i,j,k\in I\mbox{ with }i<j<k.
		\]
A subset $\underline{a}=\{a_{i}\}_{i\in I}$ is said to be \textbf{ultimately} a pseudo-Cauchy sequence if there exists $i_0\in I$ such that $\{a_i\}_{i\geq i_0}$ is a pseudo-Cauchy sequence.
\end{Def}
\begin{Obs}
If the sequence contains only two elements, then it is trivially a pseudo-Cauchy sequence. 
\end{Obs}
\begin{Obs}
The definition above is based on the one given by Kaplansky in \cite{Kapl}. There the notation $\rho<\lambda$ is used (instead of $i\in I$) where $\lambda$ is an infinite ordinal. We use the notation $i\in I$ because we think it is simpler and because we do not assume that $I$ does not have last element. The reason for our assumption is to be able to parametrize a larger class of valuations.
\end{Obs}

\begin{Obs}
In \cite{Per}, the authors work with a more general type of sequences, called \textit{pseudo-monotone}. Although they allow us to parametrize a large class of valuations, we will not use them here.
\end{Obs}

We denote by $\mathbb{S}$ the set of pseudo-Cauchy sequences in $K$. For a pseudo-Cauchy sequence $\underline{a}=\{a_i\}_{i\in I}$, if $I$ has a maximum, we denote it by $i_{\rm max}$. In this case, we denote $I^*= I\setminus\{i_{\rm max}\}$. Otherwise, we set $I^*=I$.

For each $i\in I^*$ we set
\[
\gamma_i:=v(a_{i+1}-a_i)
\]
where $i+1$ denotes the successor of $i$. For each $j\in I$, $i<j$, we have
\[
v(a_j-a_i)=\min\{v(a_j-a_{i+1}),v(a_{i+1}-a_i)\}=v(a_{i+1}-a_i)=\gamma_i.
\]
Hence, $\{\gamma_i\}_{i\in I^*}$ is an increasing sequence in $\Gamma$.

\begin{Lema}\cite[Lemma 1]{Kapl}\label{Lemmabehpseudo}
	If $\{a_i\}_{i\in I}$ is a pseudo-Cauchy sequence, then either
	\begin{equation}\label{eqnsobrepsconvseq}
		v(a_i)<v(a_j)\mbox{ for all }i,j\in I,i<j, \mbox{ (increasing)}
	\end{equation}
	or there exists $i_0\in I$ such that
	\begin{equation}\label{eqnsobrepsconvseq1}
		v(a_j)=v(a_{i_0})\mbox{ for every }j\in I, i_0\leq j \mbox{ (ultimately constant)}.
	\end{equation}
\end{Lema}

The next lemma is a very useful tool to deal with pseudo-Cauchy sequences.

\begin{Lema}\cite[Lemma 4]{Kapl}\label{lemasuperutil}
	Let $\Gamma$ be an ordered abelian group, $\beta_1,\ldots,\beta_n\in \Gamma$ and $\{\gamma_i\}_{i\in I}$ an increasing sequence in $\Gamma$, without a last element. If $t_1,\ldots,t_n$ are distinct positive integers, then there exist $h$, $1\leq h\leq n$, and $i_0\in I$, such that
	\[
	\beta_l+t_l\gamma_i>\beta_h+t_h\gamma_i\mbox{ for every }l, 1\leq l\leq n, l\neq h\mbox{ and }i\in I, i>i_0.
	\]
\end{Lema}

For example, this lemma allows us to prove the following.

\begin{Prop}\cite[Lemma 5]{Kapl}\label{Propkapl1}
	If $\{a_i\}_{i\in I}$ is a pseudo-Cauchy sequence, such that $I^*$ does not have a last element, and $f(x)\in K[x]$, then $\{f(a_i)\}_{i\in I}$ is ultimately a pseudo-Cauchy sequence.
\end{Prop}

 Hence, given a polynomial $f\in K[x]$, by Proposition \ref{Propkapl1} and Lemma \ref{Lemmabehpseudo}, $f(\underline a):=\{f(a_i)\}_{i\in I}$  is ultimately increasing or constant. If $f(\underline a)$ is ultimately constant, we denote by $\nu_{\underline{a}}(f)$ this constant value. Let
\[
S_{\underline a}:=\{f\in K[x]\mid f(\underline a)\mbox{ is ultimately constant}\}.
\]
Then we can consider the map
\[
\nu_{\underline a}:S_{\underline a}\lra \Gamma_\infty\mbox{ by }f\longmapsto \nu_{\underline a}(f).
\]

The next result follows easily from the definitions.
\begin{Lema}\label{remnkbosparcar}$\,$
	
	\noindent\textbf{(i):} If $f,g\in S_{\underline a}$, then $fg\in S_{\underline a}$ and
	\[
	\nu_{\underline{a}}(fg)=\nu_{\underline{a}}(f)+\nu_{\underline{a}}(g).
	\]
	\noindent\textbf{(ii):} If  $f, g, f+g\in S_{\underline{a}}$, then 
	\[
	\nu_{\underline{a}}(f+g)\geq \min\{\nu_{\underline{a}}(f), \nu_{\underline{a}}(g)\}.
	\]
\end{Lema}

\begin{Def}
Take a pseudo-Cauchy sequence $\underline a=\{a_i\}_{i\in I}$ such that $I^*$ does not have a maximum. Then $\underline a$ is said to be of \textbf{transcendental type} if $S_{\underline a}=K[x]$. Otherwise, $\underline a$ is said to be of \textbf{algebraic type}. 
\end{Def}

We denote by
\[
\mathbb{S}_{\rm rt}=\left\{\{a_i\}_{i\in I}\in\mathbb S\mid I^*\mbox{ has a maximum}\right\}, 
\]
and
\[
\mathbb{S}_{\rm al}=\{\underline{a}\in\mathbb S\mid \underline{a}\mbox{ is of transcendental type}\}.
\]
Take a pseudo-Cauchy sequence $\underline a$ of algebraic type. Choose a monic polynomial $F\in K[x]$ of smallest degree among all the polynomials  not fixed by $\underline a$ and set $d=\deg(F)$. Consider the family $\mathfrak v=\{\nu_i\}_{i\in I^*}$, where $\nu_i=v_{a_i,\gamma_i}$, the cut $\delta$ having
\begin{equation}\label{eqdocutmod}
\delta^L=\{\alpha\mid \alpha\leq \nu_i(F)\mbox{ for some }i\in I^*\}
\end{equation}
as the left cut set and $\gamma=\sup \delta^L\in \Lambda_\infty$ (as in Section \ref{Secincreafamil}).
We define
\[
\mathbb{S}_{\rm vt}=\{\underline{a}\in\mathbb S\mid \underline{a}\mbox{ is of algebraic type and }\gamma\neq \infty\}
\]
and 
\[
\mathbb{S}_{\rm nt}=\{\underline{a}\in\mathbb S\mid \underline{a}\mbox{ is of algebraic type and }\gamma= \infty\}.
\]
By definition
\[
\mathbb S=\mathbb{S}_{\rm rt}\sqcup \mathbb{S}_{\rm al}\sqcup \mathbb{S}_{\rm vt}\sqcup \mathbb{S}_{\rm nt}.
\]
\subsection{Pseudo-Cauchy sequences and families of valuations}

For a pseudo-Cauchy sequence $\underline a  =\{ a_i \}_{i\in I}$  and each $i\in I^*$ we consider the valuation
\[
\nu_i=v_{a_i,\gamma_i}.
\]
Then, $\{\nu_i\}_{i\in I^*}$ is a totally ordered subset of $\mathcal{V}$. Indeed, since $\{\gamma_i\}_{i\in I^*}$ is increasing and for $i,j\in I^*$ with $i<j$ we have
\[
v(a_j- a_i)= \gamma_i
\]
it follows from Lemma~\ref{lemNuaGammaMenorIgual} that
\[
\nu_i(f)\leq \nu_j(f)\mbox{ for every }f\in K[x]. 
\]
Also, for $f=x-a_j$ we have
\[
\nu_j(f)=\gamma_j>\gamma_i=\min\{\gamma_i,v(a_j-a_i)\}=\nu_i(f).
\]

\begin{Lema}\label{lemaGraumenorFixed}
Assume that $\underline a=\{a_i\}_{i\in I}\in \mathbb{S}_{\rm vt}\sqcup \mathbb{S}_{\rm al}$. Take a polynomial $g$ and assume that every polynomial of smaller degree is fixed by $\underline a$. Then for $i\in I$ large enough we have
\[
v(g(a_i)) = \nu_i(g).
\]	
\end{Lema}
\begin{proof}
	For any $i\in I$ the $x-a_i$-expansion of $g$ is
	$$g(x) = g(a_i)+\partial_1g(a_i)(x-a_i)+\ldots +\partial_rg(a_i)(x-a_i)^r. $$
	Since for every $l$, $1\leq l\leq r$, the value of $\partial_lg$ is fixed, by Lemma~\ref{lemasuperutil}
	there exists $h$, $1\leq h \leq r$, such that for $i\in I$ large enough we have
	$$v(\partial_hg(a_i))+h\gamma_i<\underset{l\neq h}{\min}\{ v(\partial_lg(a_i))+l\gamma_i \} $$
	and $v(g(a_i)) = v(\partial_hg(a_i))+h\gamma_i. $ Hence,
	\begin{align*}
		\nu_i(g) & = \underset{0\leq l \leq r}{\min}\{ v(\partial_lg(a_i))+l\gamma_i  \}\\
		& = \min\left\{ v(g(a_i)), \underset{1\leq l \leq r}{\min}\{v(\partial_hg(a_i))+h\gamma_i\}\right\}\\
		& = v(g(a_i)).
	\end{align*}
\end{proof}

\begin{Lema}\label{lemUnstableNotFixed}
Take $\underline a=\{a_i\}_{i\in I}\in \mathbb{S}_{\rm vt}\sqcup \mathbb{S}_{\rm al}$ and consider the family $\mathfrak{v}_{\underline a} = \{ \nu_i \}_{i\in I^*}$. A polynomial $f$ has the minimal degree among $\mathfrak{v}_{\underline a}$-unstable polynomials if and only if it has the smallest degree among polynomials with value not fixed by $\underline a$.
	
\end{Lema}

\begin{proof}Assume $f$ has the smallest degree $n$ among unstable polynomials. We prove by induction that every polynomial of degree smaller than $n$ (hence stable) has the value fixed by $\underline a$. Indeed, if $n=1$, then every constant polynomial is fixed by $\underline a$, hence the result follows. Suppose $g$ is such that $\deg(g)<n$ and every polynomial of smaller  degree is fixed by $\underline a$.  By  Lemma~\ref{lemaGraumenorFixed}
\[
v(g(a_i)) = \nu_i(g)\mbox{ for }i\in I^*\mbox{ large enough}.
\]
Since $g$ is stable, $v(g(a_i))$ is ultimately constant, that is,  $g$ is fixed by $\underline a$. Using again Lemma~\ref{lemaGraumenorFixed}, we have
\[
v(f(a_i)) = \nu_i(f)\mbox{ for }i\in I^*\mbox{ large enough}.
\]
Since $f$ is unstable, it follows that $f$ is not fixed by $\underline a$ and has the minimal degree with this property. 
	
For the converse, if $f$ has the smallest degree $n$  among polynomials with value not fixed by $\underline a$, then every polynomial with degree smaller than $n$ has its value fixed by $\underline a$. By Lemma~\ref{lemaGraumenorFixed},
\[
v(f(a_i)) = \nu_i(f)\mbox{ for }i\mbox{ large enough}.
\]
Therefore, $f$ is unstable and has minimal degree with this property. 
\end{proof}

\begin{Obs}\label{obsFixedStable}
The lemma above implies that  $\underline a$ is of transcendental type if and only if all polynomials are $\mathfrak{v}_{\underline a}$-stable. Moreover, in this case
\[
v(f(a_i)) = \nu_i(f)\mbox{ for }i\mbox{ large enough}.
\]
\end{Obs}

\begin{Cor}\label{corTransTypeVal}
	If $\underline{a}$ is of transcendental type, then $\nu_{\underline a}$ is a valuation on $K[x]$. Moreover,
\[
\nu_{\underline{a}}= \sup_{i\in I^*}\nu_i. 
\]
\end{Cor}

\begin{proof} Condition $\textbf{(V3)}$ follows directly from the fact that $v$ is a valuation and that $\nu_{\underline{a}}(f)=v(f)$ for every $f\in K$. If $\underline a$ is of transcendental type, then $S_{\underline a}=K[x]$ and conditions \textbf{(V1)} and \textbf{(V2)} follow from Lemma \ref{remnkbosparcar}. Therefore, $\nu_{\underline a}$ is a valuation on $K[x]$. 
	By Remark~\ref{obsFixedStable}, Proposition~\ref{lemVstable} and the definition of $\nu_{\underline{a}}$ we have $\nu_{\underline{a}}=  \displaystyle \sup_{i\in I^*} \nu_i$.
	
	\end{proof}
\begin{Obs}
The valuation $\nu_{\underline a}$ is the same as the one in Kaplansky's Theorem 2 (in \cite{Kapl}).
\end{Obs}

Take $F$ a polynomial of smallest degree not fixed by $\underline a$ and $\gamma=\sup \delta^L$ with notation as in \eqref{eqdocutmod}. The next result follows easily from the definitions and from Proposition \ref{lemSupAlgType}.
\begin{Prop}\label{corAlgTypeVal}
If $\alpha\in \Lambda_\infty$ with $\alpha\geq \gamma$, then $\mu_{F,\alpha}$ is a valuation on $K[x]$. Moreover,
\[
\mu_{F,\gamma}=\sup_{i\in I^*}\nu_i
\] 
and $\mu_{F,\gamma}$ does not depend on the choice of $F$ among polynomials of smallest degree not fixed by $\underline{a}$.
\end{Prop}

\begin{Obs}
The valuation obtained in Kaplansky's Theorem 3 (in \cite{Kapl}) corresponds to the valuation on the proposition above taking $\alpha=\infty$. Moreover, in this case, if $G$ is another monic polynomial of minimal degree not fixed by $\underline{a}$, then
\[
\mu_{F,\infty}\neq \mu_{G,\infty}.
\]
Indeed, suppose $\mu:=\mu_{F,\infty}=\mu_{G,\infty}$. Since $\mu(F)= \mu(G) = \infty$, we have
	$$\mu(F-G) \geq \min\{ \mu(F), \mu(G)  \} =\infty.$$
Since $\deg(F-G)<\deg(F)$, this implies that there exists $i_0\in I^*$ such that 
	$$
	\infty = \mu_{F,\infty}(F-G) = v((F-G)(a_i)) \text{ for every } i\in I, i>i_0.\\
	$$
This implies $(F-G)(a_i)=0$ for every $i\in I, i>i_0$. Since there are infinitely many such $i$'s, we must have $F=G$. 
\end{Obs}

We are now ready to define the map $\Psi$.
\begin{Teo}
	The function $\Psi:\mathbb S\lra \mathbb V$ given by
\[
\Psi(\underline{a})=[\mu]\mbox{ where }\mu=\sup_{i\in I^*}\nu_i
\]
is well-defined. Moreover,
	\begin{align*}
\Psi(\underline a)= \begin{cases}
			\nu_{\underline a}, &\text{ if } \underline{a} \text{ is of transcendental type},\\
			\mu_{F,\gamma}, & \text{ if } \underline{a} \text{ is of algebraic type},\\
			\nu_{i^*_{\max}}, &  \text{ if } I^* \text{ has a maximum }i^*_{\max}.
		\end{cases} 
	\end{align*}
\end{Teo}
\begin{proof}
It follows from Lemma \ref{lemNuaGammaMenorIgual}, Corollary \ref{corTransTypeVal} and Proposition \ref{corAlgTypeVal}.
\end{proof}

Write
\[
\mathbb V=\mathbb V_{\rm nt}\sqcup\mathbb V_{\rm vt}\sqcup\mathbb V_{\rm rt}\sqcup \mathbb V_{\rm al}
\]
in the obvious way. 

\begin{Cor}
We have
\[
\Psi(\mathbb S_{\rm rt})\subseteq\mathbb V_{\rm rt},\ \Psi(\mathbb S_{\rm al})\subseteq\mathbb V_{\rm al},\ \Psi(\mathbb S_{\rm vt})\subseteq\mathbb V_{\rm vt}\mbox{ and }\Psi(\mathbb S_{\rm nt})\subseteq\mathbb V_{\rm nt}.
\]
\end{Cor}
\begin{proof}
If $\underline{a}\in\mathbb S_{\rm rt}$, then
\[
\Psi(\underline{a})=\left[\nu_{i^*_{\max}}\right]=\left[v_{a_{i^*_{\max}},\gamma_{i^*_{\max}}}\right]\in \mathbb V_{\rm rt}
\]
by Lemma \ref{lemqnSobrecTransc1}.

If $\underline{a}\in\mathbb S_{\rm al}$, then
\[
\Psi(\underline{a})=\left[\nu_{\underline{a}}\right]\in \mathbb V_{\rm al}
\]
by Proposition \ref{CorVstableAlgebraic}.

Assume now that $\underline a$ is algebraic and take $F$ and $\gamma$ as above. If $\underline{a}\in \mathbb S_{\rm vt}$, then $\gamma\in \Lambda\setminus \Gamma_\Q$ and since
\[
\mu_{F,\gamma}(F)=\gamma
\]
we deduce that $\Psi(\underline a)\in \mathbb V_{\rm vt}$. Finally, if $\gamma=\infty$, then $F\in \SU(\mu_{Q,\gamma})$ and consequently $\Psi(\underline{a})\in \mathbb V_{\rm nt}$.
\end{proof}
We present now an example of a valued field for which all these subsets of $\mathbb S$ are non-empty.

\begin{Exa}
Consider the field $K=\F_p(t)^{\frac{1}{p^\infty}}$, which is the perfect hull of $\F_p(t)$ inside $\F_p\left(\left(t^\Q\right)\right)$. Let $v$ be the $t$-adic valuation on $K$. If we consider the set $I=\{1,2\}$ and
\[
a_1=t\mbox{ and }a_2=t+t^2
\]
then $\underline a=\{a_i\}_{i\in I}\in \mathbb S_{\rm rt}$ and $\Psi(\underline{a})=v_{t,2}$. Take now an element $\eta\in \F_p((t))$ transcendental over $K$. Write
\[
\eta=\sum_{j=0}^\infty b_jt^j\mbox{ with }b_j\in \F_p\mbox{ for every }j\in\N.
\]
Set $I=\N$ and for each $i\in I$ set
\[
a_i:=\sum_{j=0}^ib_jt^j\in K.
\]
Then $\underline{a}=\{a_i\}_{i\in I}$ is a pseudo-Cauchy of transcendental type, and hence $\underline{a}\in \mathbb S_{\rm al}$.

Take now
\[
\eta=\sum_{j=0}^\infty t^{p^j}=t+t^p+t^{p^2}+\ldots\in \F_p((t))\setminus K.
\]
Then $\eta$ is a root of $F=x^p-x+t\in K[x]$. If we consider the sequence
\[
a_i=\sum_{j=0}^i t^{p^j}\in K,
\]
then $\underline{a}=\{a_i\}_{i\in I}$ is a pseudo-Cauchy sequence of algebraic type over $(K,v)$ and $F$ is a monic polynomial of smallest degree not fixed by $\underline a$. Moreover, in the previous notation we have $\gamma=\infty$ and so $\underline a\in \mathbb S_{\rm nt}$.

Finally, if we take
\[
\eta=\sum_{j=0}^\infty t^{-\frac{1}{p^j}}=t^{-1}+t^{-\frac{1}{p}}+t^{-\frac{1}{p^2}}+\ldots\in \F_p\left(\left(t^\Q\right)\right)\setminus K.
\]
Then $\eta$ is a root of $F=x^p-x-t^{-p}\in K[x]$. If we consider the sequence
\[
a_i=\sum_{j=0}^i t^{-\frac{1}{p^j}}\in K,
\]
then $\underline{a}=\{a_i\}_{i\in I}$ is an algebraic pseudo-Cauchy sequence over $(K,v)$ and $F$ is a monic polynomial of smallest degree not fixed by $\underline a$. Moreover, $\gamma=0^-$ and consequently $\underline a\in \mathbb S_{\rm vt}$.
 
\end{Exa}
\subsection{The case $K=\overline K$}

If the field $K$ is algebraically closed, then the situation becomes much simpler. The next result is well-known, but we present its proof here for sake of completeness. Also, it illustrates the elements in the sets $\mathbb S_{\rm nt}$, $\mathbb S_{\rm rt}$, $\mathbb S_{\rm vt}$ and $\mathbb S_{\rm al}$.

\begin{Prop}\label{pseudcaucqysufec}
If $K$ is algebraically closed, then $\Psi$ is surjective.
\end{Prop}
\begin{proof}
Take any valuation $\nu$ on $K[x]$ extending $v$. Consider the set
\[
{\rm d}(x,K):=\{\nu(x-c)\mid c\in K\}.
\]
If ${\rm d}(x,K)$ has a largest element $\gamma=\nu(x-c)$, then $\nu=v_{c,\gamma}$. Indeed, since $K$ is algebraically closed, it is enough to show that
\[
\nu(x-b)=v_{c,\gamma}(x-b)\mbox{ for every }b\in K.
\]
If $\nu(x-b)<\nu(x-c)$, then $\nu(x-b)=v(b-c)$ and consequently
\[
v_{c,\gamma}(x-b)=\min\{\gamma,v(b-c)\}=v(b-c)=\nu(x-b).
\]
If $\nu(x-b)=\nu(x-c)$, then $\nu(x-b)\leq v(b-c)$ and consequently
\[
v_{c,\gamma}(x-b)=\min\{\gamma,v(b-c)\}=\gamma=\nu(x-b).
\]
Suppose that $\gamma\in \Gamma$, say $\gamma=v(a)$. In this case, take $I=\{1,2\}$ and
\[
a_1=c\mbox{ and }a_2=a+c.
\]
Clearly, $\underline a=\{a_i\}_{i\in I}$ is a pseudo-Cauchy sequence. It follows from the definition that
\[
\underline a\in \mathbb S_{\rm rt}\mbox{ and }[\nu]=\Psi(\underline a).
\]

If $\gamma\notin \Gamma$, then it induces a cut $\delta=\left(\Gamma_{<\gamma},\Gamma_{>\gamma}\right)$ on $\Gamma_\Q=\Gamma$. Take a cofinal set $\{\gamma_i\}_{i\in I}$ on $\delta^L$ and for each $i\in I$ an element $b_i\in K$ such that $v(b_i)=\gamma_i$. For each $i\in I$ set $a_i:=b_i+c$. Then $\underline{a}=\{a_i\}_{i\in I}$ is a pseudo-Cauchy sequence and for $f=x-c$ we have
\[
v(f(a_i))=v(a_i-c)=v(b_i)=\gamma_i.
\]
Consequently, $\underline{a}$ is of algebraic type and $x-c$ is a monic polynomial of smallest degree not fixed by $\underline{a}$. Moreover, the cut induced by $\{v(f(a_i))\}_{i\in I}$ is $\delta$. 
If $\gamma=\sup \delta^L=\infty$, then
\[
\underline a\in \mathbb S_{\rm nt}\mbox{ and }[\nu]=\Psi(\underline a).
\]
If $\gamma<\infty$, then
\[
\underline a\in \mathbb S_{\rm vt}\mbox{ and }[\nu]=\Psi(\underline a).
\]

Suppose now that ${\rm d}(x,K)$ does not have a maximum. Take a well-ordered cofinal subset $\{\gamma_i\}_{i\in I}$ of ${\rm d}(x,K)$ and for each $i\in I$ an element $a_i\in K$  such that $\gamma_i=\nu(x-a_i)$. It is easy to see that $\{a_i\}_{i\in I}$ is a pseudo-Cauchy sequence on $K$. For every $b\in K$ there exists $i\in I$ such that
\[
\nu(x-b)<\nu(x-a_j)\mbox{ for every }j\in I, i<j.
\]
Then, for $f=x-b$ we have
\begin{equation}\label{rtrnasnceiqeu}
v(f(a_j))=v(b-a_j)=\nu(x-a_j-(x-b))=\nu(x-b)\mbox{ for every }j\in I, i<j.
\end{equation}
Since $K$ is algebraically closed, this shows that the value of every polynomial in $K[x]$ is fixed by $\underline a$, and so $\underline a$ is of transcendental type. This and \eqref{rtrnasnceiqeu} imply that
\[
\underline{a}\in \mathbb S_{\rm al}\mbox{ and }[\nu]=\Psi(\underline a).
\]
\end{proof}
\section{Approximation types}\label{approcxsec}

Let $\mathfrak{B}=\mathfrak{B}(K,v)$ be the set of all closed balls in $K$ with radii in $\Gamma=\Gamma_v$. A \textbf{nest of balls} $\{B_i\}_{i\in I}$ is a family of elements $B_i\in \mathfrak{B}$ such that $I$ is a totally ordered set of indices satisfying
\[
i<j \Llr B_i\supset B_j.
\]
For a nest of balls $\mathcal{B}=\{B_i\}_{i\in I}$,  let
$$\overline{\mathcal{B}}:= \{ B\in \mathfrak{B}\mid B_i\subseteq B \text{ for some } i\in I  \}. $$

\begin{Def}
	An \textbf{approximation type} over $(K,v)$ is a nest of balls $\Appr$ such that $\boldsymbol{ \rm A}=\overline{\boldsymbol{\rm A}}$.
\end{Def}

Let $\mathbb{A}$ be the set of all approximation types over $(K,v)$. 

\begin{Lema}\cite[Lemma 5.1]{NartJosnei}\label{lemUnicoAppr}
For every nest of balls $\mathcal{B}$, there exists a uniquely determined approximation type $\Appr$ such that $\Appr$ and $\mathcal{B}$ are cofinal in each other. 
\end{Lema}
\begin{proof}
It is easy to show that $\overline{\mathcal B}$ is this approximation type.
\end{proof}
The \textbf{support} of $\Appr$ is the set
$$\supp(\Appr):=\{\gamma\in \Gamma_v\mid \Appr \text{ contains a closed ball of radius } \gamma\}. $$
This is an initial segment of $\Gamma_v$. By \cite[Lemma 2.16]{KuhlmannApprTypes}, for each $\gamma \in \supp(\Appr)$ there exists a unique ball in $\Appr$ of radius $\gamma$ which we denoted by  $\Appr_\gamma$.

Fix an extension $\overline v$ of $v$ to $\overline K$. For each $\gamma\in \Gamma$ and $a\in K$ we denote by $B_{\overline K}(a,\gamma)$ the corresponding ball in $\overline K$, i.e.,
\[
B_{\overline K}(a,\gamma)=\{b\in \overline K\mid \overline v(b-a)\geq \gamma\}.
\] 

Take a pseudo-Cauchy sequence $\underline{a}=\{a_i\}_{i\in I}$. We associate to $\underline{a}$ the nest of balls $\mathcal B=\{B(a_i, \gamma_i)\}_{i\in I^*}$ and the approximation type
	$$\iota(\underline a) := \overline{\mathcal B}.$$
It is easy to see that
\[
\SU(\iota(\underline a))=\bigcup_{i\in I^*}\Gamma_{\leq \gamma_i}.
\]
Denote
\[
\tilde{\underline{a}}:=\bigcap_{i\in I^*}B_{\overline K}(a_i,\gamma_i).
\] 
\begin{Prop}
The map
\[
\iota:\mathbb{S}\lra \mathbb{A},\ \underline{a}\mapsto\iota(\underline{a})
\]
is surjective. Moreover we have the following.
\begin{description}
\item[(i)] $\underline a\in \mathbb S_{\rm rt}$ if and only if $\SU(\iota(\underline{a}))$ admits a maximum.

\item[(ii)] $\underline a\in \mathbb S_{\rm vt}$ if and only if $\SU(\iota(\underline{a}))$ does not admit a maximum,
\[
\tilde{\underline{a}}\neq \emptyset\mbox{ and }\SU(\iota(\underline{a}))\neq \Gamma.
\]

\item[(iii)] $\underline a\in \mathbb S_{\rm nt}$ if and only if $\SU(\iota(\underline{a}))$ does not admit a maximum,
\[
\tilde{\underline{a}}\neq \emptyset\mbox{ and }\SU(\iota(\underline{a}))= \Gamma.
\] 

\item[(iv)] $\underline a\in \mathbb S_{\rm al}$ if and only if $\tilde{\underline{a}}=\emptyset$.
\end{description}
\end{Prop}

\begin{proof}
Take an approximation type $\Appr\in \mathbb A$. Suppose first that $\SU(\Appr)$ has a maximum $\gamma$. Then $\Appr=\overline{ \Appr_\gamma}$ and $\Appr_\gamma=B(a,\gamma)$ for some $a\in K$.  Take $b\in K$ such that $v(b)=\gamma$. Setting $I=\{1,2\}$, $a_1=a$ and $a_2=a_1+b$ we have
\[
\underline a=\{a_i\}_{i\in I}\in \mathbb S_{\rm rt}\mbox{ and }\iota(\underline a)=\Appr.
\]
Moreover, it is clear that if $\underline a\in \mathbb S_{\rm rt}$, then $\SU(\iota(\underline a))$ has a maximum.

Suppose now that $\SU(\Appr)$ does not have a maximum. Take a cofinal sequence $\{\gamma_i\}_{i\in I}$ in $\SU(\Appr)$ and denote $\Appr_i:=\Appr_{\gamma_i}$. Since $\SU(\Appr)$ does not have a maximum, we can assume that $\Appr_i\supsetneq \Appr_{i+1}$. For each $i\in I$ choose an element $a_i\in \Appr_i\setminus \Appr_{i+1}$. We claim that $\{a_i\}_{i\in I}$ is a pseudo-Cauchy sequence on $K$. Indeed, for $i,j,k\in I$ with $i<j<k$, by our choice, we have
\[
a_i\in \Appr_i\setminus \Appr_j\mbox{ and }a_k\in \Appr_j. 
\]
Since $a_i\in \Appr_i$ and $a_j\in \Appr_j$, we have
\[
\Appr_i=B(a_i,\gamma_i)\mbox{ and }\Appr_j=B(a_j,\gamma_j).
\]
The fact that $a_i\notin \Appr_j$ and $a_k\in \Appr_j$ imply that
\[
v(a_i-a_j)<\gamma_j\leq v(a_k-a_j),
\]
which shows that $\underline{a}$ is a pseudo-Cauchy sequence. Moreover, it is easy to see that $\Appr=\iota(\underline a)$ and hence $\iota$ is surjective.

It remains to show the items \textbf{(ii)}, \textbf{(iii)} and \textbf{(iv)}. Assume that $\SU(\iota(\underline{a}))$ does not admit a maximum and
\begin{equation}\label{eqasngdslc}
B':=\bigcap_{i\in I^*}B_{\overline K}(a_i,\gamma_i)\neq \emptyset.
\end{equation}
Take $c\in B'$ and let $F$ be the minimal polynomial of $c$ over $K$. Let $c=c_0,c_1,\ldots,c_s\in \overline K$ be the roots of $F$. Write
\[
F=(x-c) \prod_{l=1}^s(x-c_l).
\]
For each $l$, $1\leq l\leq s$, the set $\{a_i-c_l\}_{i\in I}$ is a pseudo-Cauchy sequence, so its value is increasing or ultimately constant. Since $c\in B'$, we deduce that 
\[
v(a_i-c)=v(a_i-a_{i+1}+a_{i+1}-c)=\min\{\gamma_i,\gamma_{i+1}\}=\gamma_i
\]
and consequently $\{\overline v(a_i-c)\}_{i\in I}$ is increasing. Since
\[
v(F(a_i))=\overline v\left((a_i-c) \prod_{l=1}^s(a_i-c_l)\right)=\overline{v}(a_i-c)+\sum_{l=1}^s\overline{v}(a_i-c_l),
\]
we deduce that $\{v(F(a_i))\}_{i\in I}$ is ultimately increasing. This shows that $\underline{a}$ is of algebraic type. In this case, by definition 
\[
\underline{a}\in \mathbb S_{\rm vt}\Llr\SU(\iota(\underline a))\neq \Gamma,
\]
and this is same as saying that
\[
\underline{a}\in \mathbb S_{\rm nt}\Llr\SU(\iota(\underline a))=\Gamma.
\]

Finally, if $B'=\emptyset$, then for every monic polynomial $F\in K[x]$ we write
\[
F=\prod_{l=1}^r (x-c_l)\mbox{ with }c_1,\ldots,c_l\in \overline K.
\]
Since $c_l\notin B'$, there exists $i\in I$ such that $c_l\notin B_{\overline K}(a_i, \gamma_i)$, that is, $\overline v(c_l-a_i)<\gamma_i$. This means that for $j>i$ we have
\[
\overline v(c_l-a_j)=\overline v(c_l-a_i+a_i-a_j)=\overline v(c_l-a_i)
\]
and consequently $\{\overline{v}(a_i-c_l)\}_{i\in I}$ is ultimately constant. Since
\[
v(F(a_i))=\sum_{l=1}^r\overline v(a_i-c_l),
\]
we deduce that $\{v(F(a_i))\}_{i\in I}$ is ultimately constant. This shows that $\underline{a}\in \mathbb S_{\rm al}$ and this concludes the proof.
\end{proof}

Set
\begin{displaymath}
\begin{array}{rcl}
\mathbb A_{\rm rt}&:=&\{\Appr\in \mathbb A\mid \SU(\Appr)\mbox{ has a maximum}\},\\
\mathbb A_{\rm vt}&:=&\{\Appr\in \mathbb A\mid \SU(\Appr)\mbox{ does not have a maximum, }\tilde{\underline{a}}\neq \emptyset\mbox{ and }\SU(\Appr)\neq \Gamma\},\\
\mathbb A_{\rm nt}&:=&\{\Appr\in \mathbb A\mid \SU(\Appr)\mbox{ does not have a maximum, }\tilde{\underline{a}}\neq \emptyset\mbox{ and }\SU(\Appr)= \Gamma\},\\
\mathbb A_{\rm al}&:=&\{\Appr\in \mathbb A\mid \SU(\Appr)\mbox{ does not have a maximum and }\tilde{\underline{a}}=\emptyset\}.

\end{array}
\end{displaymath}

\begin{Cor}
We have
\[
\mathbb A=\mathbb A_{\rm rt}\sqcup\mathbb A_{\rm vt}\sqcup\mathbb A_{\rm nt}\sqcup\mathbb A_{\rm al}
\]
and
\[
\iota\left(\mathbb S_{\rm rt}\right)=\mathbb A_{\rm rt},\iota\left(\mathbb S_{\rm vt}\right)=\mathbb A_{\rm vt}, \iota\left(\mathbb S_{\rm nt}\right)=\mathbb A_{\rm nt} \mbox{ and }\iota\left(\mathbb S_{\rm al}\right)=\mathbb A_{\rm al}.
\]
\end{Cor}
For an approximation type
\[
\Appr=\{B(a_i,\gamma_i)\mid \gamma_i \in \SU(\Appr)\}
\]
we define
\[
\Phi(\Appr)=[\mu]\mbox{ where }\mu=\sup\{\nu_i\mid \gamma_i\in \SU(\Appr)\}
\]
where $\nu_i:=v_{a_i,\gamma_i}$.

\begin{Teo}\label{teoPhiAppr}
The map $\Phi:\mathbb A\lra \mathbb V$ given by $\Appr\mapsto \Phi(\Appr)$ is well-defined, injective and
\[
\Psi=\Phi\circ \iota.
\]
\end{Teo}
\begin{proof}
It follows from the definitions that $\Phi$ is well-defined and $\Psi=\Phi\circ \iota$. In order to show that $\Phi$ is injective take two approximation types $\Appr$ and $\Appr'$ such that $\Phi(\Appr)=\Phi(\Appr')$. We will show that $\Appr=\Appr'$.

If $\SU(\Appr)$ admits a maximum $\gamma$ ($\Appr_\gamma=B(a,\gamma)$), then
\[
\Phi(\Appr)=[v_{a,\gamma}].
\]
Since $\Phi(\Appr)=\Phi(\Appr')$, this can only happen if $B(a,\gamma)\in \Appr'$. Since $\overline {\Appr'}=\Appr'$, we deduce that $\Appr\subseteq \Appr'$. If $\Appr\neq \Appr'$, then there would exist $B(a',\gamma')\in \Appr'$ with $B(a',\gamma')\subsetneq B(a,\gamma)$. By Lemma \ref{lemNuaGammaMenorIgual} this would imply that
\[
v_{a,\gamma}<v_{a',\gamma'}\leq \sup\{v_{a_i,\gamma_i}\mid \gamma_i\in \SU(\Appr')\}
\]
and this is a contradiction to $\Phi(\Appr)=\Phi(\Appr')$.

Assume now that $\SU(\Appr)$ does not admit a maximum. Hence $\SU(\Appr')$ also does not have a maximum. Take $B(a,\gamma)\in \Appr$. Then $\gamma\in \SU(\Appr')$ and consequently, there exists $\gamma'\in \SU(\Appr')$ with $\gamma'>\gamma$. Set
\[
\Appr'_{\gamma'}=B(a',\gamma').
\]
For
\[
\mu:=\sup\{v_{a_\gamma,\gamma}\mid \gamma\in \SU(\Appr)\}=\sup\{v_{a_{\gamma'},\gamma'}\mid \gamma'\in \SU(\Appr')\},
\]
Proposition \ref{treescufutue} implies that
\begin{equation}\label{treeodrtetola}
\left]-\infty, \mu\right[\mbox{ is totally ordered.}
\end{equation}
This and Lemma \ref{lemasobrebola} \textbf{(iv)} imply that $v_{a',\gamma'}\geq v_{a,\gamma}$. By Lemma \ref{lemNuaGammaMenorIgual} we deduce that
\[
\Appr_\gamma=B(a,\gamma)\supseteq B(a',\gamma')=\Appr'_{\gamma'}.
\]
Since $\Appr'$ is an approximation type, this implies that $B(a,\gamma)\in \Appr'$. Hence $\Appr\subseteq \Appr'$. The other inclusion follows by a symmetric reasoning.
\end{proof}
\subsection{The case $K=\overline K$}
Given a valuation $\mu$ on $K(x)$, the set
$$\apprv := \{  B\cap K\mid B=B(x, \gamma)\in \mathfrak{B}(K(x),\nu) , \gamma\in \Gamma_v  \} $$
is an approximation type over $(K,v)$ (\cite[Lemma 3.4]{KuhlmannApprTypes}). 

If we consider $[\nu]\in \mathbb{V}$, one can show  that the definition of $\apprv$  is independent of the choice of the representative.

\begin{Prop}\cite[Theorems 1.2 and 1.3]{KuhlmannApprTypes} \label{TeoKuhlmannApprTypesSurjective}
The map 
\begin{align*}
		\mathbb{V} &\longrightarrow \mathbb{A}\\
		[\nu] &\longmapsto \apprv
\end{align*}
is surjective. If $K$ is algebraically closed, then this mapping is bijective. 
\end{Prop}
\begin{Prop}
If $K$ is algebraically closed, then $\Phi$ is a bijection. Moreover, its inverse is the map in Proposition \ref{TeoKuhlmannApprTypesSurjective}.
\end{Prop}
\begin{proof}
That $\Phi$ is a bijection follows from Proposition \ref{pseudcaucqysufec} and Proposition \ref{teoPhiAppr}. The second part is easy to check.
\end{proof}

\end{document}